\documentclass{amsproc}

\usepackage{amsmath}
\usepackage{amssymb}
\usepackage{amsthm}
\usepackage{mathrsfs}

\newtheorem{theorem}{Theorem}[section]

\theoremstyle{definition}
\newtheorem{definition}[theorem]{Definition}
\newtheorem{proposition}[theorem]{Proposition}

\newtheorem{corollary}[theorem]{Corollary}

\theoremstyle{remark}
\newtheorem{remark}[theorem]{Remark}

\numberwithin{equation}{section}

\begin{document}

\title[The Montel theorem for polynomial dynamics]{An alternative proof of the non-Archimedean Montel theorem for polynomial dynamics}


\author{LEE, Junghun}
\address{Graduate School of Mathematics, Nagoya University, 
Nagoya 464-8602, Japan}
\email{m12003v@math.nagoya-u.ac.jp}
\thanks{The author would like to thank to the scientific committee and the referees. He also thanks Professor Charles Favre for his comment. He also deeply appreciates to Professor Tomoki Kawahira.}



\date{18 March, 2014}

\begin{abstract}
We will see an alternative proof of the non-Archimedean Montel theorem for polynomial dynamics.
\end{abstract}

\maketitle


\section{Introduction}

P. Montel has proved that a family of holomorphic functions defined on an open set in the field of complex numbers is normal if the family is uniformly bounded in \cite{Mont16}.
In the light of the Arzel\`a-Ascoli theorem, we can restate Montel's theorem in terms of equicontinuity.
It has played a crucial role in complex analysis.
Among its many applications, we focus on the applications to dynamical systems generated by iterations of functions.
One of the main topics in \textit{the theory of complex dynamical systems}, which investigates the iterations of a given rational map over the field of complex numbers, is to determine whether or not a point is in the \textit{Fatou set}, which is defined as the largest open set on which the iterations of a given rational map are equicontinuous.
Montel's theorem can be a useful criterion to achieve it.
Moreover, Montel's theorem can be applied to investigate a number of basic properties of the \textit{Julia set}, which is defined as the complement of the Fatou set. We will see the properties of the Julia set in Section $2$.
Furthermore, combined with some basic knowledge of the theory of complex dynamical systems, we can also show the density of repelling periodic points in the Julia set.
See \cite{Miln06} for more details on the theory of complex dynamical systems.

Not only in the theory of complex dynamical systems, but also in the \textit{theory of non-Archimedean dynamical systems}, which is a theory of dynamical systems of the projective line over algebraically closed, complete, and non-Archimedean fields and rational maps over the field, a criterion such as Montel's theorem is also important to find the Fatou set.
See \cite{Silv07} for more details on the theory of non-Archimedean dynamical systems.

L-C. Hsia has originally proved the non-Archimedean Montel theorem, which is also called Hsia's criterion, and has applied it to show that the Julia set is contained in the topological closure of the set of periodic points in \cite{Hs00}.
Moreover, C. Favre, J. Kiwi, and E. Trucco have proved several versions of Montel's theorem in a non-Archimedean setting and have applied them to dynamics on the Berkovich projective and affine line in \cite{FKT12}.

The aim of this paper is to give an alternative proof of Hsia's criterion for polynomial dynamics by using non-Archimedean Green functions, see Section $2$ of this paper for the explicit statement.

In the theory of complex dynamical systems, B. Branner and J. H. Hubbard introduced Green functions in \cite{BH83} to investigate the orbit of critical points of polynomial maps over the field of complex numbers.
On the other hand, S. Kawaguchi and J. H. Silverman introduced non-Archimedean Green functions and proved that the Green function is closely related to the Fatou set in \cite{KS09}.
Our new proof uses two fundamental results on Green functions, namely \cite[Theorem $6$]{KS09} and \cite[Proposition $3$]{KS09}.
We will see a rigorous definition of non-Archimedean Green functions and some properties in Section $3$.

\section{Equicontinuity and Montel's theorems}

In this section, we will state Montel's theorem for polynomial dynamics.
Let us begin with the definition of equicontinuity.

\begin{definition}[Equicontinuity]\label{2.1}
Let $X$ be a metric space with a metric $d$, $U$ be an open set, and $f : X \rightarrow X$ be a continuous map. We say that $f$ is \textit{equicontinuous} on $U$ if for every $x$ in $U$ and every $\epsilon > 0$, the element $x$ has an open neighborhood $V = V_x$ such that for any $y$ in $V_x$ and any $k$ in $\{ 0, 1, \cdots, \}$, we have
$$
d(f^k(y), f^k(x)) < \epsilon.
$$
\end{definition}

Let $K$ be the field $\mathbb{C}$ of complex numbers with the Euclidean norm or an algebraically closed field $\mathbb{K}$ with a complete and non-Archimedean norm and $|\cdot|$ be the norm on $K$.
We denote the \textit{projective line}, which is defined as the union of $K$ and infinity, over $K$ by $\mathbb{P}^1_K$ and consider it as a metric space with respect to the chordal metrics defined as follows:

\begin{definition}[Chordal metric]\label{2.2}
The \textit{chordal metric} is defined by 
\begin{align*}
\rho(z, w) := 
\begin{cases}
\displaystyle{ |z - w | \over ||z|| \cdot ||w|| } \quad &(z, w \in K), \\\\
\displaystyle{ 1 \over ||z|| }  \quad &(z \in K, w = \infty).
\end{cases}
\end{align*}
on $\mathbb{P}^1_K$ where 
\begin{align*}
||z|| :=
\begin{cases}
\sqrt{1 + |z|^2} \quad &(z \in \mathbb{C}), \\\\
\max \{ 1, |z| \} \quad &(z \in \mathbb{K}).
\end{cases}
\end{align*} 
\end{definition}

Now we state Montel's theorems for polynomial dynamics.

\begin{theorem}[Montel's theorem for polynomial dynamics]\label{2.3}
Let $U$ be an open subset of $K$ and $f$ be a polynomial map over $K$ of degree $d \geq 2$.
Suppose that there exist two distinct elements $\alpha$ and $\beta$ in $K$ such that
$$
\bigcup_{k = 0}^{\infty}f^{k}(U) \cap \{ \alpha, \beta \} = \emptyset.
$$
Then, the polynomial map $f$ is equicontinuous on $U$ with respect to the chordal metric on $\mathbb{P}^1_K$.
Moreover, if $K = \mathbb{K}$, then it is also true even when $\alpha = \beta$. 
\end{theorem}

A proof of Theorem $2.3$ is given in \cite[Theorem $3.7$]{Miln06} and \cite[Theorem $2.2$]{Hs00} for complex and non-Archimedean case respectively.
In this paper, we give a new proof of Theorem \ref{2.3} for the case $K = \mathbb{K}$.
We remark that the original statement holds in a broader case than Theorem \ref{2.3}.

Let us see some applications of Theorem \ref{2.3} to the theory of dynamical systems.

\begin{corollary}\label{2.4}
Let $f$ be a polynomial map over $K$ of degree $d \geq 2$.
Then, the following $1$ to $4$ hold:
\begin{enumerate}
\item The Julia set of $f$ has empty interior if the Fatou set is non-empty.\\
\item The Julia set of $f$ is uncountable if the Julia set is non-empty.\\
\item The Julia set of $f$ has no isolated point.\\
\item The backward orbit of any point in the Julia set is dense in the Julia set.
\end{enumerate}
\end{corollary}

\begin{proof}[Proof of Corollary \ref{2.4}]
See \cite[Corollary $4.11$, $4.13$, $4.14$, and $4.15$]{Miln06} for the case $K = \mathbb{C}$ and see \cite[Corollary $5.32$]{Silv07} for the case $K = \mathbb{K}$.
\end{proof}

\begin{remark}\label{2.5}
If $K = \mathbb{C}$, it is known that the Julia set is always non-empty if the rational map over $K$ has degree grater than $1$, see \cite[Lemma $4.8$]{Miln06} for more details.
On the other hand, if $K = \mathbb{K}$, it is easy to check that there is a rational map with empty Julia set even when the degree of the rational map is greater than $1$.
\end{remark}

\begin{remark}\label{2.6}
Corollary \ref{2.4} holds for the case when $f$ is a rational map of degree $d \geq 2$.
\end{remark}

\section{Non-Archimedean Green functions}

In this section, we will see the definition and properties of non-Archimedean Green functions.
As in \cite[equation ($2$)]{KS09}, we define non-Archimedean Green functions as follows.

\begin{proposition}[Green function]\label{3.1}
Let $f$ be a polynomial map over $\mathbb{K}$ of degree $d \geq 2$.
Then the \textit{Green function} of $f$
\begin{equation*}
\begin{split}
G_f : \mathbb{K} &\rightarrow \mathbb{R} \\
z &\mapsto \lim_{n \rightarrow \infty} {1 \over d^n} \cdot \log(||f^n(z)||) - \log(||z||)
\end{split}
\end{equation*}
is well-defined.
Moreover, the Green function can be written as
$$
G_{f} (z) = \sum_{k = 0}^{\infty} {1 \over d^k}\log \left({ || f^{k + 1}(z) ||^{1 / d} \over || f^{k}(z) || } \right)
$$
for any $z$ in $\mathbb{K}$.
\end{proposition}

\begin{proof}[Proof of Proposition \ref{3.1}]
See \cite[Proposition $3$]{KS09} for the proofs.
\end{proof}

The following is a one dimensional version of \cite[Theorem $6$]{KS09}.

\begin{theorem}\label{3.2}
Let $f$ be a polynomial map over $\mathbb{K}$ of degree $d \geq 2$.
Then, the Green function $G_f$ of $f$ is H\"older continuous. 
Moreover, the following statements are equivalent:
\begin{enumerate}
\item The Green function $G_f$ is locally constant on an open neighborhood of $z$.
\item  The polynomial map $f$ is equicontinuous on an open neighborhood of $z$ with respect to the chordal metric.
\end{enumerate}
\end{theorem}

\begin{proof}[Proof of Theorem \ref{3.2}]
See \cite[Theorem $6$]{KS09} for the proof.
\end{proof}

\begin{remark}
Theorem \ref{3.2} implies that a point is contained in the Fatou set if and only if the Green function is locally constant at the point.
\end{remark}

\section{The proof of Theorem \ref{2.3}}

Let us close this paper with proof of Theorem \ref{2.3} when $K = \mathbb{K}$.

\begin{proof}[Proof of Theorem \ref{2.3}]
It is sufficient to prove the case when $\alpha = \beta = 0$.

Suppose that
$$
U = \overline{D}(a, r) := \{ z \in \mathbb{K} \mid |z - a| \leq r \}
$$
where $a \in \mathbb{K}$ and $r \in \mathbb{R}_{> 0}$.
Then, it follows from \cite[Proposition $5.16$]{Silv07} that
$$
f(\overline{D}(a, r)) = \overline{D}(f(a), s)
$$
where $s := \max\{ |f(z) - f(a)| \mid z \in U \} > 0$.
This implies that $f^k(U)$ is a closed disk in $\mathbb{K}$ centered at $f^k(a)$.
Moreover, since $f^k(U) \cap \{ 0\} = \emptyset$ and $f^k(U)$ is a disk in $\mathbb{K}$, we have
$$
|f^k(z)| = |f^k(z) - f^k(a) + f^k(a)| = \max\{ |f^k(z) - f^k(a)|, |f^k(a)| \} = |f^k(a)|
$$
for any $z \in U$ and any $k \in \{0, 1, \cdots \}$.
This implies that
\begin{equation*}
\begin{split}
{ || f^{k + 1}(z) ||^{1 / d} \over || f^k(z) || } 
= { || f^{k + 1}(a) ||^{1 / d} \over || f^k(a) || }
\end{split}
\end{equation*}
for any $z \in U$ and any $k \in \{0, 1, \cdots \}$.
By Proposition \ref{3.1}, we have 
$$
G_{f} (z) = \sum_{k = 0}^{\infty} {1 \over d^k}\log \left({ || f^{k + 1}(z) ||^{1 / d} \over || f^k(z) || } \right) = \sum_{k = 0}^{\infty} {1 \over d^k}\log \left( { || f^{k + 1}(a) ||^{1 / d} \over || f^k(a) || } \right) = G_f (a)
$$
for any $z \in U$.
Thus, the Green function of $f$ is constant on $U$ so, by Theorem \ref{3.2}, the polynomial map $f$ is equicontinuous on $U$ with respect to the chordal metric.
\end{proof}


\bibliographystyle{amsalpha}


\end{document}